\documentclass[letterpaper, 10pt, conference]{ieeeconf}

\usepackage{cite}
\usepackage{amsmath}
\usepackage{amsfonts}
\usepackage{amssymb}
\usepackage{mathtools}
\usepackage{graphicx}
\usepackage{float}
\usepackage{tabularx}
\usepackage{color}
\usepackage{transparent}
\usepackage{algorithm}
\usepackage{algorithmic}

\renewcommand{\bar}{\overline}
\renewcommand{\hat}{\widehat}
\renewcommand{\tilde}{\widetilde}

\newcommand{\Tr}{\mathrm{Tr}}
\newcommand{\R}{\mathbb{R}}

\newcommand{\N}{N}

\newcommand{\tends}{\rightarrow}
\newcommand{\indenti}[1]{\transparent{0}{#1}\transparent{1}}

\newtheorem{proposition}{Proposition}
\newtheorem{assumption}{Assumption}

\newtheorem{lemma}{Lemma}

\newtheorem{remark}{Remark}
\newtheorem{theorem}{Theorem}

\IEEEoverridecommandlockouts
\title{\LARGE \bf
Discrete-Time Adaptive Control of a Class of Nonlinear Systems Using High-Order Tuners
}

\author{Peter A. Fisher and Anuradha M. Annaswamy
\thanks{This work was supported by the Boeing Strategic University Initiative}
\thanks{The authors are with the Department of Mechanical Engineering, Massachusetts Institute of Technology, Cambridge, MA, 02139 USA}%
\thanks{Corresponding author: P.A. Fisher, \ {\tt\small pafisher@mit.edu}}%
}

\begin{document}

\maketitle
\thispagestyle{empty}
\pagestyle{empty}

\begin{abstract}

This paper concerns the adaptive control of a class of discrete-time nonlinear systems with all states accessible. Recently, a high-order tuner algorithm was developed for the minimization of convex loss functions with time-varying regressors in the context of an identification problem. Based on Nesterov's algorithm, the high-order tuner was shown to guarantee bounded parameter estimation when regressors vary with time, and to lead to accelerated convergence of the tracking error when regressors are constant. In this paper, we apply the high-order tuner to the adaptive control of a particular class of discrete-time nonlinear dynamical systems. First, we show that for plants of this class, the underlying dynamical error model can be causally converted to an algebraic error model. Second, we show that using this algebraic error model, the high-order tuner can be applied to provably stabilize the class of dynamical systems around a reference trajectory.

\end{abstract}

\section{INTRODUCTION} \label{sec:introduction}

Adaptive control problems take the form of controlling a plant containing unknown parameters, which requires simultaneous online learning and control \cite{Narendra2005,Annaswamy2021hist,Landau11,Goodwin_1984}. The field is rich with numerous applications and a theoretical history stretching back decades \cite{Annaswamy2021hist}. As autonomous systems become more and more pervasive, there is a growing need for faster learning and faster control. The many approaches to adaptive control that have been developed over the years can be roughly divided into two categories \cite{Goodwin_1984}: indirect adaptive control, in which the unknown plant parameters are learned and state feedback is calculated from the estimates; and direct adaptive control, in which the state feedback is directly learned.

Many recent approaches have taken an indirect approach. The approaches in \cite{Dean_2018,Dean_2018a} are one illustration of indirect adaptive control of LTI systems, where the unknown parameters are first estimated using a least squares approach followed by a system-level synthesis method to determine the resulting LQR gain $K$. It should be noted that indirect adaptive control has a very rich history prior to \cite{Dean_2018,Dean_2018a} as well \cite{Narendra2005,Annaswamy2021hist,Landau11,Goodwin_1984}.

Indirect approaches, however, have the requirement that the initial parameter estimate is sufficiently close, and they require a persistently exciting input so that the parameter estimates converge to their true values -- \cite{Dean_2018}, for example, calls for Gaussian noise as input.
Direct adaptive control algorithms, on the other hand, determine a control structure wherein the parameters are directly adjusted based on a suitable performance error derived using a reference model. Often the adaptive laws for adjusting these parameters are based on an error model \cite{Narendra2005} that leads to a stable adaptive law. The main advantage of this approach over the indirect one is that there is no requirement related to persistent excitation. As one cannot always guarantee that such an excitation is present, and as it is often counter to the system performance goals, this direct approach can be advantageous in many cases. For the most part, the adaptive laws for updating the parameter estimates are based on a gradient descent approach, both in continuous time \cite{Narendra2005,Annaswamy2021hist} and discrete-time \cite{Goodwin_1984}.

We restrict our attention in this paper to the matched uncertainty setting. There has been previous literature on the unmatched uncertainty setting, especially the hybrid MRAC approach in \cite{quindlen2015}, which relies other methods such as concurrent learning \cite{chowdhary2013} and composite learning \cite{pan2016} to avoid the need for persistent excitation. However, all of the above methods rely on perfect parameter learning, which requires at least a guarantee of finite excitation. Another hybrid MRAC approach is presented in \cite{joshi2019} which does not place any assumption on excitation level, but does assume a bounded state and time derivative. Additionally, all of the papers above focus only on continuous time.

Within direct adaptive control, high-order tuners represent a more recent departure from gradient descent-based methods. High-order tuners for adaptive control were first studied in \cite{Morse_1992}. Within the past few years, a discrete-time high-order tuner was developed in \cite{gaudio2020accelerated} for parameter learning with time-varying regressors. Developed from a well-known theory of 2nd-order gradient algorithms for accelerated convergence \cite{Polyak_1964,Nesterov_1983,Nesterov_2018,Wibisono_2016}, the high-order tuner algorithm in \cite{gaudio2020accelerated} was shown to lead to faster learning than gradient descent-based methods, as well as strong non-asymptotic convergence guarantees for constant regressors. Crucially, the discrete-time high-order tuner is provably stable when regressors vary with time. Its distinct advantage is accelerated convergence of the output error: it has been shown in \cite{gaudio2020accelerated} that when regressors are constant, the high-order tuner has convergence guarantees that are a log factor away from those of Nesterov's algorithm \cite{Nesterov_2018} and that are significantly faster than those of gradient descent algorithms. Additionally, in \cite{Gaudio2019a}, a continuous-time version of the high-order tuner algorithm was shown in simulation to result in an accelerated convergence of the output error to zero.

The discrete-time high-order tuner discussed above has only been studied in the context of system identification, and has employed algebraic error models for parameter learning. In this paper, we consider the adaptive control problem for a class of feedback-linearizable dynamical systems whose states are accessible. For this class, we show that the high-order tuner can be applied, leading to global stability and convergence of the underlying tracking error in the state to zero. As in \cite{scitech_paper}, a causal filtering approach based on \cite{Kudva_Narendra1974} converts the underlying dynamical error model into an algebraic error model. Unlike in \cite{Lion1967,Kreisselmeier1977,Slotine1989,Duarte1989,Ortega2016_DREM} where the underlying states are filtered as well, the approach used here only generates an augmented error signal, as in \cite{Kudva_Narendra1974}. Using this error model and a high-order tuner, we show that the class of dynamic systems can be adaptively controlled in a stable manner (see \cite{scitech_paper} for a few preliminary results).


The main contribution of this paper is the application of the high-order tuner in \cite{gaudio2020accelerated} to direct adaptive control of a class of feedback-linearizable systems. We prove that this algorithm guarantees global boundedness of the closed-loop adaptive system and asymptotic tracking of a reference model regardless of the level of excitation in the input or the initial parameter estimate.
Our proof technique is straightforward and generalizable to a broad class of laws for updating the parameter estimate.

To the authors' knowledge, ours is the first paper to apply the high-order tuner in \cite{gaudio2020accelerated} to general discrete-time adaptive control. Our paper complements \cite{scitech_paper}, which explores a simplified high-order tuner under noisy disturbances, and \cite{Cui_preprint}, which establishes parameter learning for identification problems in discrete-time dynamical systems with persistent excitation.

The paper proceeds as follows. Section \ref{sec:problem_setting} lays out the problem setting, describes the framework by which we convert the dynamical error model to an algebraic error model, and introduces useful notation for the subsequent proofs. Section \ref{sec:gd_adaptive_control} provides an illustrative example of our proof technique on a gradient descent-based adaptive law. Section \ref{sec:ht_adaptive_control} presents the main result of our paper: a proof of stability using the high-order tuner in \cite{gaudio2020accelerated} as an adaptive law. Section \ref{sec:simulations} presents simulation results of the high-order tuner's performance on a simple common dynamical system. Finally, Section \ref{sec:conclusion} provides concluding remarks, and the Appendix contains all proofs in the paper as well as a discussion of allowable high-order tuner hyperparameters.
\section{PROBLEM SETTING} \label{sec:problem_setting}






The problem that we consider in this paper is the adaptive control of the states-accessible plant
\begin{equation} \label{eqn:plant}
    x_{p(k+1)} = A_px_{pk} + B\left(\sum_{i = 1}^p a_if_i(x_{pk}) + u_k\right)
\end{equation}
where $A_p \in \R^{n \times n}$ and all $a_i \in \R^m$ are unknown, while $B \in \R^{n \times m}$ and all $f_i : \R^n \to \R$ are known, subject to the following assumptions:
\begin{assumption} \label{asn:plant_assumptions}
    We assume that
    \begin{enumerate}
        \item the pair $(A_p, B)$ is controllable and all columns of $B$ are linearly independent, and
        \item each $f_i(\cdot)$ is known and globally $M_i$-Lipschitz with $f_i(0) = 0$.
    \end{enumerate}
\end{assumption}
The goal is to determine the control input $u_k$ in real time such that $x_{pk}$ behaves in a desired manner.

A standard procedure in adaptive control is to choose a reference model
\begin{equation} \label{eqn:reference_model}
    x_{m(k+1)} = A_mx_{mk} + Br_k
\end{equation}
where $A_m$ is chosen by the control designer to be Schur-stable with the desired closed-loop eigenvalues, and $r_k \in \R^m$ is a reference input with $\|r_k\| \leq r_{max}$ chosen such that $x_{mk}$ follows the desired trajectory of the plant. For realizability, given that $A_m$ is chosen without a priori knowledge of $A_p$, the following matching condition is a standard assumption employed in adaptive control:
\begin{assumption} \label{asn:matching_condition}
    We assume that there exists some $K_* \in \R^{m \times n}$ such that
    \begin{equation} \label{eqn:matching_condition}
        A_m = A_p + BK_*.
    \end{equation}
\end{assumption}

If all parameters were known, choosing $u_k$ of the form
\begin{equation} \label{eqn:optimal_feedback}
    u_k = K_*x_{pk} + r_k - \sum_{i = 1}^p a_if_i(x_{pk}),
\end{equation}
a feedback linearizing controller, would ensure that the closed-loop plant response follows the same trajectory as the reference model.

It is well-known that an adaptive control input of the form
\begin{equation} \label{eqn:feedback}
    u_k = \hat{K}_kx_{pk} + r_k - \sum_{i = 1}^p \hat{a}_if_i(x_{pk}),
\end{equation}
where $\hat{K}_k$ and $\hat{a}_{ik}$ are estimates of the unknown parameters $K_*$ and $a_i$ in \eqref{eqn:optimal_feedback}, can guarantee that the state error
\begin{equation} \label{eqn:e}
    e_k := x_{pk} - x_{mk}
\end{equation}
converges to zero if the estimates are suitably adjusted using an adaptive law \cite{Goodwin_1984,Narendra2005}.

In the remainder of this section, we propose a new general algorithm for this adaptive control problem, based on results in \cite{Kudva_Narendra1974} for system identification. In Section \ref{sec:ht_adaptive_control}, we then propose the addition of the high-order tuner developed in \cite{gaudio2020accelerated,Cui_preprint} as a particular adaptive law.

The certainty equivalence input in \eqref{eqn:feedback} is equivalent to
\begin{equation} \label{eqn:feedback_rearranged}
    u_k = K_*x_{pk} + r_k - \sum_{i = 1}^p a_if_i(x_{pk}) + \tilde{\Theta}_k\phi_k
\end{equation}
where
\begin{gather}
    \hat{\Theta}_k := [\hat{K}_k, \hat{a}_{1k}, \dots, \hat{a}_{pk}] \label{eqn:Theta} \\
    \Theta_* := [K_*, a_1, \dots, a_p] \label{eqn:Theta_star} \\
    \tilde{\Theta}_k := \hat{\Theta}_k - \Theta_* \label{eqn:Theta_tilde} \\
    \phi_k^\top := [x_{pk}^\top, -f_1(x_{pk}), \dots, -f_p(x_{pk})]. \label{eqn:phi}
\end{gather}
The closed-loop adaptive system for the plant in \eqref{eqn:plant} with the controller in \eqref{eqn:feedback_rearranged} is thus described by
\begin{equation} \label{eqn:plant_rearranged}
    x_{p(k+1)} = A_mx_{pk} + Br_k + B\tilde{\Theta}_k\phi_k.
\end{equation}
It is easy to see that \eqref{eqn:plant_rearranged}, \eqref{eqn:reference_model}, and \eqref{eqn:e} yield the error model
\begin{equation} \label{eqn:error_model_2}
    e_{k+1} = A_me_k + B\tilde{\Theta}_k\phi_k.
\end{equation}

\subsection{An Equivalent Algebraic Error Model}
\label{subsec:error_model_1}

Equation \eqref{eqn:error_model_2} is a dynamical error model, as it relates the two main errors, the state error $e_k$ and the parameter error $\tilde{\Theta}_k$, through a dynamical model. Our approach based on \cite{Kudva_Narendra1974} transforms this problem into an algebraic error model of the form
\begin{equation} \label{eqn:error_model_1}
    \varepsilon_{k+1} = \tilde{\Theta}_k\phi_k
\end{equation}
where \cite{Kudva_Narendra1974}
\begin{gather}
    \varepsilon_{k+1} := (B^\top B)^{-1}B^\top(e_{k+1} - A_me_k), \label{eqn:prediction_error}.
\end{gather}
It should be noted that the prediction error $\varepsilon_{k+1}$ is another performance metric that depends on the state error through the relation
\begin{equation} \label{eqn:error_model_2_1}
    e_{k+1} = A_me_k + B\varepsilon_{k+1}.
\end{equation}

It should also be noted that a causal adaptive law can be derived for adjusting the parameter estimate $\hat{\Theta}_k$ defined in \eqref{eqn:Theta} by first measuring the state and the state errors on the right-hand side of \eqref{eqn:prediction_error} and then updating the parameter estimate. This overall algorithm is summarized in Algorithm \ref{alg:adaptive_control} \cite{scitech_paper}. In line \ref{algln:adaptive_law}, ADAPT refers to any iterative algorithm for updating $\hat{\Theta}_k$.

\begin{algorithm}[b]
  \caption{General Algorithm for Direct Adaptive Control of the Plant in \eqref{eqn:plant} \cite{scitech_paper}}
  \label{alg:adaptive_control}
  \begin{algorithmic}[1]
    \STATE {\bfseries Input:} initial conditions $x_{p0}$, $x_{m0}$, initial parameter estimate $\hat{\Theta}_0$, reference input $\{r_k\}_{k \geq 0}$, reference system $(A_m, B)$, functions $f_1(\cdot), \dots, f_p(\cdot)$, hyperparameters
    \FOR{$k = 0, 1, 2, \ldots$}
    \STATE \textbf{Receive} reference input $r_k$
    \STATE Let $\phi_k = [x_{pk}^\top, -f_1(x_{pk}), \dots, -f_p(x_{pk})]^\top$
    \STATE Let $u_k = \hat{\Theta}_k\phi_k + r_k$
    \STATE \textbf{Apply} $u_k$ to plant
    \STATE \textbf{Measure} new state $x_{p(k+1)}$
    \STATE \textbf{Simulate} $x_{m(k+1)} \gets A_mx_{mk} + Br_k$
    \STATE Let $e_{k+1} = x_{p(k+1)} - x_{m(k+1)}$
    \STATE Let $\varepsilon_{k+1} = (B^\top B)^{-1}B^\top(e_{k+1} - A_me_k)$
    \STATE \textbf{Update} $\hat{\Theta}_{k+1} \gets$ ADAPT$(\hat{\Theta}_k, \phi_k, \varepsilon_{k+1}, *)$ \label{algln:adaptive_law}
    \ENDFOR
  \end{algorithmic}
\end{algorithm}

As we shall show in the following sections, update laws based on the gradient of a loss function - in particular, normalized gradient descent and the high-order tuner - lead to global stability, regardless of the level of excitation in the input or the initial parameter estimate $\hat{\Theta}_0$.


\subsection{Preliminaries} \label{subsec:preliminaries}




In this section, we first review a well-known result pertaining to the Lyapunov stability of linear time-invariant systems:
\begin{proposition}[\cite{hespanha2018}] \label{pro:dlyap}
    For any matrix $A \in \R^{n \times n}$, the following conditions are equivalent:
    \begin{enumerate}
        \item $A$ is Schur-stable, i.e. all eigenvalues of $A$ are inside the unit circle.
        \item For every symmetric positive-definite $Q$, there exists a unique symmetric positive-definite $P$ satisfying the discrete-time Lyapunov equation
    \end{enumerate}
    \begin{equation} \label{eqn:dlyap}
        A^\top PA - P = -Q.
    \end{equation}
\end{proposition}

Finally, we provide a useful result pertaining to the relative growth rates of $x_{pk}$ and $\phi_k$:
\begin{lemma} \label{lem:phi_O_x}
    Let $\phi_k$ be the regressor defined in \eqref{eqn:phi}. Then, under Assumption \ref{asn:plant_assumptions}, there exists a known constant $C > 0$ such that $\|\phi_k\|^2 \leq C\|x_{pk}\|^2$.
\end{lemma}
We omit the proof, as it is fairly straightforward.
\section{A FIRST-ORDER APPROACH TO ADAPTIVE CONTROL} \label{sec:gd_adaptive_control}


The algebraic error model in \eqref{eqn:error_model_1} lends itself easily to a loss function given by \cite{gaudio2020accelerated}
\begin{align}
    L_k(\hat{\Theta}_k) &= \frac{1}{2}\|\varepsilon_{k+1}\|^2 = \frac{1}{2}\|\tilde{\Theta}_k\phi_k\|^2. \label{eqn:loss_function}
\end{align}
Using \eqref{eqn:error_model_1}, the gradient of the loss function can then be calculated as
\begin{equation} \label{eqn:L_gradient}
    \nabla L_k(\hat{\Theta}_k) = \tilde{\Theta}_k\phi_k\phi_k^\top = \varepsilon_{k+1}\phi_k^\top.
\end{equation}
It is easy to see that $L_k(\hat{\Theta}_k)$ is non-strongly convex and has a time-varying smoothness parameter of $\|\phi_k\|^2$. We therefore use a normalized loss function given by \cite{gaudio2020accelerated}
\begin{equation} \label{eqn:loss_normalized}
    \bar{f}_k(\hat{\Theta}_k) = \frac{L_k(\hat{\Theta}_k)}{\N_k}
\end{equation}
with the normalization term
\begin{equation} \label{eqn:N}
    \N_k = \max\{\mu, \|\phi_k\|^2\}
\end{equation}
for some $\mu > 0$. It is thus apparent that $\bar{f}_k(\hat{\Theta}_k)$ is convex and 1-smooth, with a gradient that can be calculated as
\begin{equation} \label{eqn:f_gradient}
    \nabla\bar{f}_k(\hat{\Theta}_k) = \frac{1}{\N_k}\varepsilon_{k+1}\phi_k^\top.
\end{equation}
The problem of minimizing the loss function in \eqref{eqn:loss_normalized} leads naturally to the multivariable form of the well-known normalized gradient descent adaptive law given by \cite{Goodwin_1984}
\begin{equation} \label{eqn:gd_adaptive_law}
    \hat{\Theta}_{k+1} = \hat{\Theta}_k - \gamma\nabla\bar{f}_k(\hat{\Theta}_k).
\end{equation}

\subsection{Stability of the Gradient Descent Adaptive Law} \label{subsec:gd_stability}
We now show that Algorithm \ref{alg:adaptive_control} with \eqref{eqn:N}-\eqref{eqn:gd_adaptive_law} in place of ADAPT on line \ref{algln:adaptive_law} is a globally stable adaptive controller. The first step is to quantify the evolution of the parameter error, which is addressed in Proposition \ref{pro:gd_Lyapunov}.
\begin{proposition} \label{pro:gd_Lyapunov}
    The adaptive law in \eqref{eqn:N}-\eqref{eqn:gd_adaptive_law} results in a bounded parameter error $\tilde{\Theta}_k$ for all $k$ if $\mu > 0$ and $0 < \gamma < 2$ with
    \begin{equation} \label{eqn:gd_Lyapunov_function}
        V_k = \|\tilde{\Theta}_k\|_F^2
    \end{equation}
    as a Lyapunov function.
\end{proposition}
\begin{proof}
    See Subsection \ref{subsec:gd_Lyapunov_proof} in the Appendix.
\end{proof}

We now prove the convergence of the overall closed loop adaptive system described by the error model in \eqref{eqn:error_model_2}:
\begin{theorem} \label{thm:gd_error_bounded}
  The closed-loop adaptive system defined by \eqref{eqn:plant}, \eqref{eqn:feedback}, \eqref{eqn:error_model_2}, \eqref{eqn:prediction_error}, and \eqref{eqn:N}-\eqref{eqn:gd_adaptive_law} with $\mu > 0$ and $0 < \gamma < 2$ results in $\lim_{k \to \infty} \|e_k\| = 0$.
\end{theorem}
\begin{proof}
    See Subsection \ref{subsec:gd_stability_proof} in the Appendix.
\end{proof}

\begin{remark}
    Proposition \ref{pro:gd_Lyapunov} and Theorem \ref{thm:gd_error_bounded} are both well-known in the existing literature -- see e.g. \cite{Ioannou_2006}. We include both proofs in order to make apparent the similarity with the high-order tuner algorithm in the next section.
\end{remark}
\begin{remark}
    Theorem \ref{thm:gd_error_bounded} together with Proposition \ref{pro:gd_Lyapunov} extends the results of \cite{Kudva_Narendra1974} to show that the approach based on an algebraic error model can be used to guarantee closed-loop stability for discrete-time adaptive control.
\end{remark}

Another interesting point to note is the generality of the proof. As will become evident from our discussions in the next section, the same method of proof can be employed for any adaptive law that can guarantee the property $\lim_{k \to \infty} \frac{\|\varepsilon_{k+1}\|^2}{\N_k} = 0$.
In this case, this property followed from the structure of the Lyapunov increment (see \eqref{eqn:gd_Lyapunov_increment} in the Appendix).

\section{ADAPTIVE CONTROL WITH A HIGH-ORDER TUNER} \label{sec:ht_adaptive_control}




We now state the main result of this paper. For the plant given in \eqref{eqn:plant}, we propose the adaptive controller given in Algorithm \ref{alg:adaptive_control}, with Algorithm \ref{alg:high_order_tuner} in place of ADAPT in line \ref{algln:adaptive_law}. Algorithm \ref{alg:high_order_tuner} summarizes the high-order tuner adaptive law \cite{gaudio2020accelerated}.

\begin{algorithm}[b]
  \caption{ADAPT (Projected High Order Tuner) \cite{gaudio2020accelerated}}
  \label{alg:high_order_tuner}
  \begin{algorithmic}[1]
    \STATE {\bfseries Input:} time step $k$, current parameter estimate $\hat{\Theta}_k$, regressor $\phi_k$, prediction error $\varepsilon_{k+1}$, gains $\mu$, $\gamma$, $\beta$
    \IF{$k == 0$}
    \STATE $\hat{\Xi}_0 \gets \hat{\Theta}_0$
    \ELSE
    \STATE \textbf{Receive} $\hat{\Xi}_k$ from previous iteration
    \ENDIF
    \STATE Let $N_k = \max\{\mu, \|\phi_k\|^2\}$
    \STATE Let $\nabla\bar{f}_k(\hat{\Theta}_k) = \frac{1}{\N_k}\varepsilon_{k+1}\phi_k^\top$
    \STATE Let $\bar{\Theta}_{k} = \hat{\Theta}_k - \gamma\beta\nabla\bar{f}_k(\hat{\Theta}_k)$
    \STATE \textbf{Update} $\hat{\Theta}_{k+1} \gets \bar{\Theta}_{k} - \beta(\bar{\Theta}_{k} - \hat{\Xi}_k)$
    \STATE Let $\nabla\bar{f}_k(\hat{\Theta}_{k+1}) = \frac{1}{\N_k}((\hat{\Theta}_{k+1} - \hat{\Theta}_k)\phi_k + \varepsilon_{k+1})\phi_k^\top$
    \STATE \textbf{Update} $\hat{\Xi}_{k+1} \gets \hat{\Xi}_k - \gamma\nabla\bar{f}_k(\hat{\Theta}_{k+1})$
    \STATE \textbf{Return} $\hat{\Theta}_{k+1}$
  \end{algorithmic}
\end{algorithm}

The specific updates that constitute the high-order tuner are given by
\begin{align}
  \hat{\Xi}_{k+1} &= \hat{\Xi}_k - \gamma\nabla\bar{f}_k(\hat{\Theta}_{k+1}) \label{eqn:ht_Xi} \\
  \bar{\Theta}_{k} &= \hat{\Theta}_k - \gamma\beta\nabla\bar{f}_k(\hat{\Theta}_k) \label{eqn:ht_Theta_bar} \\
  \hat{\Theta}_{k+1} &= \bar{\Theta}_{k} - \beta(\bar{\Theta}_{k} - \hat{\Xi}_k) \label{eqn:ht_Theta}
\end{align}
where $\hat{\Xi}_k$ is an auxiliary parameter estimate, $\nabla\bar{f}_k(\hat{\Theta}_k) = \frac{\nabla L_k(\hat{\Theta}_k)}{\N_k}$ and $\nabla\bar{f}_k(\hat{\Theta}_{k+1}) = \frac{\nabla L_k(\hat{\Theta}_{k+1})}{\N_k}$, and the gradients of $L_k$ are given by
\begin{align}
  \nabla L_k(\hat{\Theta}_k) &= \tilde{\Theta}_k\phi_k\phi_k^\top = \varepsilon_{k+1}\phi_k^\top \label{eqn:ht_L_gradient_k} \\
  \nabla L_k(\hat{\Theta}_{k+1}) &= \tilde{\Theta}_{k+1}\phi_k\phi_k^\top \nonumber \\
  &= ((\hat{\Theta}_{k+1} - \hat{\Theta}_k)\phi_k + \varepsilon_{k+1})\phi_k^\top \label{eqn:ht_L_gradient_k+1}.
\end{align}
When the regressors are constants (i.e. $\phi_k = \phi =$ constant), \eqref{eqn:ht_Xi}-\eqref{eqn:ht_Theta} reduce to Nesterov's algorithm \cite{gaudio2020accelerated},
\begin{equation} \label{eqn:nesterov}
    \hat{\Theta}_{k+1} = \hat{\Theta}_k + \bar{\beta}(\hat{\Theta}_k - \hat{\Theta}_{k-1}) - \bar{\gamma}\nabla L(\hat{\Theta}_k + \bar{\beta}(\hat{\Theta}_k - \hat{\Theta}_{k-1}))
\end{equation}
where
\begin{equation} \label{eqn:nesterov_gains}
    \bar{\beta} = 1 - \beta, \,\, \bar{\gamma} = \gamma\beta.
\end{equation}
\eqref{eqn:ht_Xi}-\eqref{eqn:ht_Theta} are therefore a high-order counterpart to the adaptive law in \eqref{eqn:gd_adaptive_law}, and include momentum components (second term in \eqref{eqn:nesterov}) and acceleration components (third term in \eqref{eqn:nesterov}).


\subsection{Stability of the High-Order Tuner Adaptive Law} \label{subsec:ht_stability}

As before, we first quantify the evolution of the parameter error $\tilde{\Theta}_k$ and the auxiliary parameter estimate $\hat{\Xi}_k$ in the following proposition.
\begin{proposition} \label{pro:ht_Lyapunov}
    The adaptive law in \eqref{eqn:N} and \eqref{eqn:ht_Xi}-\eqref{eqn:ht_L_gradient_k+1} results in a bounded parameter error $\tilde{\Theta}_k$ and a bounded auxiliary parameter estimate $\hat{\Xi}_k$ for all $k$ if $\mu > 0$, $0 < \beta < 2$, $0 < \gamma < \sqrt{\frac{2 - \beta}{\beta}}$, and $\alpha > 0$ as defined in \eqref{eqn:ht_alpha} with
    \begin{equation} \label{eqn:ht_Lyapunov_function}
        V_k = \|\hat{\Xi}_k - \Theta_*\|_F^2 + \|\hat{\Theta}_k - \hat{\Xi}_k\|_F^2
    \end{equation}
    as a Lyapunov function.
\end{proposition}
\begin{proof}
    See Subsection \ref{subsec:ht_Lyapunov_proof} in the Appendix.
\end{proof}

The main result of this paper, that the high-order tuner algorithm accomplishes the control objective of $\|e_k\| \to 0$ as $k \to \infty$, is now given in the following theorem.
\begin{theorem} \label{thm:ht_error_bounded}
  For the plant given in \eqref{eqn:plant}, Algorithm \ref{alg:adaptive_control} with Algorithm \ref{alg:high_order_tuner} as ADAPT and $\mu > 0$, $0 < \beta < 2$, $0 < \gamma < \sqrt{\frac{2 - \beta}{\beta}}$, and $\alpha > 0$ as defined in \eqref{eqn:ht_alpha} results in $\lim_{k \to \infty} \|e_k\| = 0$.
\end{theorem}
\begin{proof}
    See Subsection \ref{subsec:ht_stability_proof} in the Appendix.
\end{proof}
\begin{remark}
    In Proposition \ref{pro:ht_Lyapunov} and Theorem \ref{thm:ht_error_bounded} (as in Proposition \ref{pro:gd_Lyapunov} and Theorem \ref{thm:gd_error_bounded}), we make no assumptions on the level of excitation in the input or on the initial parameter estimate $\hat{\Theta}_0$. Therefore, this adaptive law can be applied with any bounded input $\{r_k\}_{k \geq 0}$ and any initial parameter estimate.
\end{remark}
\section{SIMULATION RESULTS} \label{sec:simulations}

To show that the high-order tuner achieves state tracking performance that is comparable to or better than that of gradient descent, we conducted simulations of a simple plant as in \eqref{eqn:plant} using Algorithm \ref{alg:adaptive_control}. We compare the results using \eqref{eqn:N}-\eqref{eqn:gd_adaptive_law} and Algorithm \ref{alg:high_order_tuner} as ADAPT in line \ref{algln:adaptive_law}.

\subsection{Simulation Details}

As in \cite{scitech_paper}, the numerical experiments were conducted using the linearized short-period dynamics of a transport aircraft flying at a low altitude at 250 ft/s, taken from Exercise 1.2 in \cite{Lavretsky2013}. We add an integral error state $\dot{q}_e = q - r$ so that the pitch rate tracks a command signal $r$ with zero steady-state error, assume a discrete-time controller with a 100 Hz sampling rate, and discretize the resulting dynamics using a zero-order hold to obtain the nominal discrete-time dynamics
\begin{equation} \label{eqn:simulation_plant}
    x_{p(k+1)} = Ax_{pk} + bu_k + b_rr_k
\end{equation}
where $x_{pk} = \left[\alpha, q, q_e\right]^\top$. Further details can be found in \cite{scitech_paper}.

We considered the reference model
\begin{equation}
    x_{m(k+1)} = A_mx_{mk} + b_rr_k = (A + b\theta_{LQR}^\top)x_{mk} + b_rr_k
\end{equation}
 where $\theta_{LQR}$ is the gain matrix obtained from LQR on the nominal discrete-time dynamics with cost matrices $Q = \mathrm{diag}(\left[0, 0, 1\right]^\top)$ and $R = 1$. We then assumed a parametric uncertainty such that $A_p = A_m - b\theta_*^\top$ for some unknown $k_*$ and applied the certainty equivalence control input
\begin{equation}
    u_k = \hat{\theta}_k^\top x_{pk}
\end{equation}
with $\hat{\theta}_0 = \theta_{LQR}$. The resulting error model was given by
\begin{equation}
    e_{k+1} = A_me_k + b\tilde{\theta}_k^\top\phi_k = A_me_k + b\varepsilon_{k+1}
\end{equation}
with $\phi_k = x_{pk}$.

\subsection{Results and Discussion}

A Monte Carlo simulation was conducted with 2000 trials. In each trial, $\theta_*$ was obtained by multiplying each element of $\theta_{LQR}$ by an i.i.d random value uniformly distributed over $[-0.5, 2]$. The adaptive control task was to track the reference model with $r_k = 5\ \forall k \geq 0$ starting from $x_{p0} = x_{m0} = 0$. For simplicity, we chose $\mu = 1$. Hyperparameter tuning was carried out to ensure the fastest possible convergence of $\|e_k\|$ to zero. For gradient descent we chose $\gamma = 1$, corresponding to the well-known projection algorithm \cite{Goodwin_1984}. For the high-order tuner, we found that choosing $\gamma$ as large as possible for any given value of $\beta$ (see Subsection \ref{subsec:ht_hyperparameters} in the Appendix) led to fastest reduction of both $\|e_k\|$ and $|\varepsilon_k|$.

Figures \ref{fig:e_simulation} and \ref{fig:epsilon_simulation} show the results of our simulations. Solid lines are mean values over all trials, and the darker and lighter windows around them are 50\% and 90\% confidence intervals, respectively. We find that under both performance metrics, the high-order tuner performs comparably to gradient descent. Intriguingly, however, the high-order tuner tends to produce slightly larger values of $|\varepsilon_k|$ and slightly smaller values of $\|e_k\|$. This result runs counter to the intuition provided by \eqref{eqn:error_model_2_1}.

Further research is needed to understand this result, as well as to understand why it appears best to choose $\gamma$ as large as possible for any given $\beta$. One possible source of intuition may be the reduction to Nesterov's algorithm in \eqref{eqn:nesterov}-\eqref{eqn:nesterov_gains} It is possible that having an extra hyperparameter allows the adaptive law to be somewhat tuned to the particular dynamical system. It is also possible that there exists another Lyapunov function besides the one in \eqref{eqn:ht_Lyapunov_function}, which could provide more clarity.

\begin{figure}
    \centering
    \includegraphics[width=0.5\textwidth]{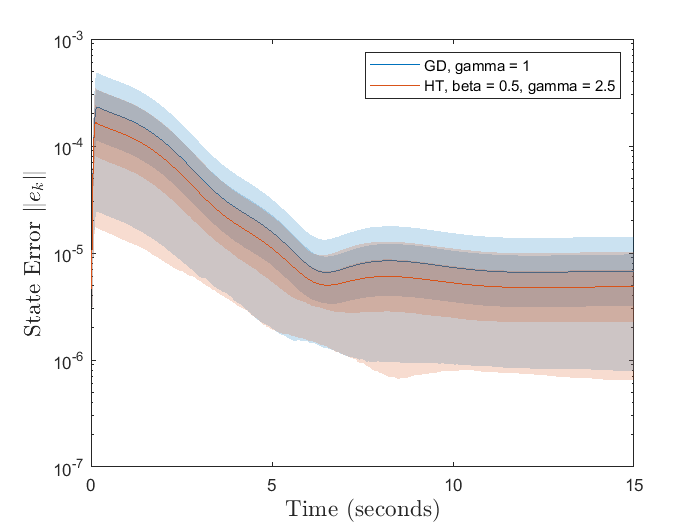}
    \caption{State error $\|e_k\|$ of the simulation results of both adaptive laws. During each trial, the unknown parameter $\theta_*$ was obtained by multiplying each element of $\theta_{LQR}$ by an i.i.d random variable in $\mathcal{U}(-0.5, 2)$. Mean values over all trials and 50\% and 90\% confidence intervals are plotted.}
    \label{fig:e_simulation}
\end{figure}
\begin{figure}
    \centering
    \includegraphics[width=0.5\textwidth]{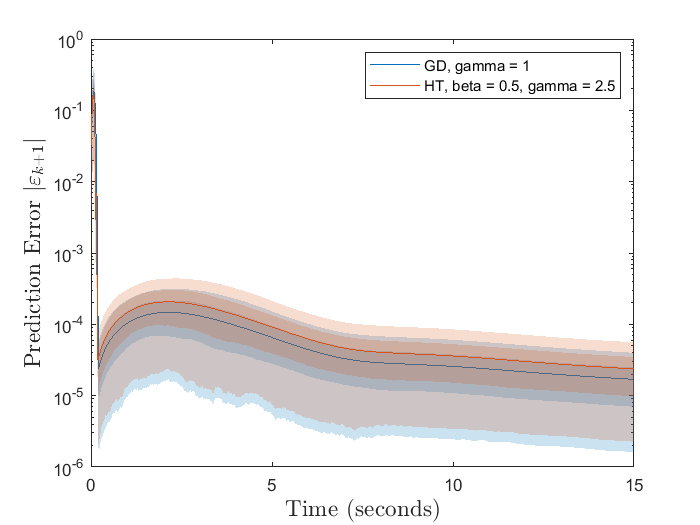}
    \caption{Prediction error $|\varepsilon_{k+1}|$ of the simulation results of both adaptive laws. During each trial, the unknown parameter $\theta_*$ was obtained by multiplying each element of $\theta_{LQR}$ by an i.i.d random variable in $\mathcal{U}(-0.5, 2)$. Mean values over all trials and 50\% and 90\% confidence intervals are plotted.}
    \label{fig:epsilon_simulation}
\end{figure}
\section{CONCLUSIONS AND FUTURE WORKS} \label{sec:conclusion}


In this paper, we present a novel algorithm for model-reference adaptive control of the class of nonlinear systems given in \eqref{eqn:plant}. This algorithm uses a causal filtering method to convert the resulting dynamical error model into an algebraic error model and applies the discrete-time high-order tuner presented in \cite{gaudio2020accelerated} as the adaptive law. Crucially, the algorithm is shown to guarantee that $\|e_k\| \to 0$ as $k \to \infty$ using a simple and general proof method for systems with all states accessible. In this proof, we make no assumptions on the initial parameter estimate or the amount of excitation in the input.

We also provide simulation results showing that the high-order tuner achieves comparable or slightly better performance than the standard gradient descent-based adaptive law. Simulations in \cite{Gaudio2019a} have shown accelerated convergence of the state error $e_k$ to zero using a continuous-time equivalent of the high-order tuner algorithm. Further research is needed to understand the influence of the choice of gains on the performance of the high-order tuner, and to explore how this accelerated convergence might be realized in discrete time.


\bibliographystyle{IEEEtran}
\bibliography{IEEEabrv,References.bib}
\appendix

Proofs of all propositions and theorems are provided below, as well as a discussion of allowable high-order tuner hyperparameters. While the proof of Proposition \ref{pro:gd_Lyapunov} is well-known, we include it for completeness. The proof of Proposition \ref{pro:ht_Lyapunov} is based on the proof of Theorem 4 in \cite{gaudio2020accelerated}, but differs in several crucial steps that are needed for the proof of Theorem \ref{thm:ht_error_bounded}. Finally, as the bulk of the proof of Theorem \ref{thm:gd_error_bounded} can be applied verbatim to Theorem \ref{thm:ht_error_bounded}, Subsection \ref{subsec:ht_stability_proof} only details the differences.

\subsection{Proof of Proposition \ref{pro:gd_Lyapunov}} \label{subsec:gd_Lyapunov_proof}


    Consider the candidate Lyapunov function given in \eqref{eqn:gd_Lyapunov_function}.
    Applying \eqref{eqn:f_gradient}-\eqref{eqn:gd_adaptive_law} and rearranging, we obtain
    \begin{flalign*}
        &\Delta V_k = \|\tilde{\Theta}_{k+1}\|_F^2 - \|\tilde{\Theta}_k\|_F^2 && \\
        &\indenti{\Delta V_k} \leq \|\hat{\Theta}_{k+1}' - \Theta_*\|_F^2 - \|\tilde{\Theta}_k\|_F^2 && \\
        &\indenti{\Delta V_k} = \|\tilde{\Theta}_k - \frac{\gamma}{\N_k}\varepsilon_{k+1}\phi_k^\top\|_F^2 - \|\tilde{\Theta}_k\|_F^2 && \\
        &\indenti{\Delta V_k} = \Tr[(\tilde{\Theta}_k - \frac{\gamma}{\N_k}\varepsilon_{k+1}\phi_k^\top)^\top(\tilde{\Theta}_k - \frac{\gamma}{\N_k}\varepsilon_{k+1}\phi_k^\top)] && \\
        &\indenti{\Delta V_k =} - \Tr[\tilde{\Theta}_k^\top\tilde{\Theta}_k] && \\
        &\indenti{\Delta V_k} = \Tr[-\frac{\gamma}{\N_k}\tilde{\Theta}_k^\top\varepsilon_{k+1}\phi_k^\top - \frac{\gamma}{\N_k}\phi_k\varepsilon_{k+1}^\top\tilde{\Theta}_k && \\
        &\indenti{\Delta V_k = \Tr.} + \frac{\gamma^2}{\N_k^2}\phi_k\varepsilon_{k+1}^\top\varepsilon_{k+1}\phi_k^\top] && \\
        &\indenti{\Delta V_k} = -\frac{2\gamma}{\N_k}\varepsilon_{k+1}^\top\tilde{\Theta}_k\phi_k + \frac{\gamma^2\|\phi_k\|^2}{\N_k^2}\|\varepsilon_{k+1}\|^2 && \\
        &\indenti{\Delta V_k} = -\frac{\gamma}{\N_k}(2 - \frac{\gamma\|\phi_k\|^2}{\N_k})\|\varepsilon_{k+1}\|^2.
    \end{flalign*}
    Finally, noting that $\N_k \geq \|\phi_k\|^2$ by \eqref{eqn:N}, we are left with
    \begin{equation} \label{eqn:gd_Lyapunov_increment}
        \Delta V_k \leq -\gamma(2 - \gamma)\frac{\|\varepsilon_{k+1}\|^2}{\N_k} \leq 0
    \end{equation}
    if $0 < \gamma < 2$.
\subsection{Proof of Theorem \ref{thm:gd_error_bounded}} \label{subsec:gd_stability_proof}

Consider the Lyapunov function in \eqref{eqn:gd_Lyapunov_function}. By Proposition \ref{pro:gd_Lyapunov}, we know that $V_k \geq 0$ and $\Delta V_k \leq 0$. It follows immediately that $V_k$ is bounded, and that
\begin{gather}
    0 \leq \lim_{k \tends \infty} V_k \leq V_0 \implies \nonumber \\
    0 \leq V_0 + \sum_{k = 0}^\infty \Delta V_k \leq V_0 \implies \nonumber \\
    -V_0 \leq \sum_{k = 0}^\infty \Delta V_k \leq 0 \implies \nonumber \\
    \lim_{k \tends \infty} \Delta V_k = 0.
\end{gather}
Substituting \eqref{eqn:gd_Lyapunov_increment}, we get
\begin{equation} \label{eqn:gd_lim_goes_to_0}
    \lim_{k \tends \infty} \frac{\|\varepsilon_{k+1}\|^2}{\N_k} = 0.
\end{equation}

Note that \eqref{eqn:gd_lim_goes_to_0} with \eqref{eqn:N} implies that if $\|\phi_k\|$ is bounded, then $\lim_{k \tends \infty} \varepsilon_k = 0$, which as stated earlier implies that $\lim_{k \tends \infty} \|e_k\| = 0$. The final step is to show that $\|\phi_k\|$ is bounded.
Lemma \ref{lem:phi_O_x} together with \eqref{eqn:gd_lim_goes_to_0} implies that
\begin{equation}
    \lim_{k \to \infty} \frac{\|\varepsilon_{k+1}\|^2}{\max\{\mu, C\|x_{pk}\|^2\}} = 0,
\end{equation}
which in turn implies that either (1) $\lim_{k \to \infty} \|\varepsilon_k\| = 0$ or (2) $\lim_{k \to \infty} \frac{\|\varepsilon_{k+1}\|}{\|x_{pk}\|} = 0$. If case (1) is satisfied, then the proof is complete. Thus, we henceforth assume case (2).

Now consider the function $V_k^x = x_{pk}^\top Px_{pk}$ for a $P = P^\top > 0$ which satisfies \eqref{eqn:dlyap} for $A_m$ and any $Q = Q^\top > 0$. Using \eqref{eqn:plant_rearranged}, \eqref{eqn:error_model_1}, and \eqref{eqn:dlyap}, the increment of $V_k^x$ is given by
\begin{flalign}
    &\Delta V_k^x = -x_{pk}^\top Qx_{pk} + r_k^\top B^\top PBr_k + \varepsilon_{k+1}^\top B^\top PB\varepsilon_{k+1} && \nonumber \\
    &\indenti{\Delta V_k^x =} + 2x_{pk}^\top A_m^\top PBr_k + 2x_{pk}^\top A_m^\top PB\varepsilon_{k+1} && \nonumber \\
    &\indenti{\Delta V_k^x =} + 2r_k^\top B^\top PB\varepsilon_{k+1} && \nonumber \\
    &\indenti{\Delta V_k^x} \leq -\lambda_{min}(Q)\|x_{pk}\|^2 + \|B^\top PB\|_2r_{max}^2 && \nonumber \\
    &\indenti{\Delta V_k^x =} + \|B^\top PB\|_2\|\varepsilon_{k+1}\|^2 + 2\|A_m^\top PB\|_2\|x_{pk}\|r_{max} && \nonumber \\
    &\indenti{\Delta V_k^x =} + 2\|A_m^\top PB\|_2\|x_{pk}\|\|\varepsilon_{k+1}\| && \nonumber \\
    &\indenti{\Delta V_k^x =} + 2\|B^\top PB\|_2r_{max}\|\varepsilon_{k+1}\| && \nonumber \\
    &\textstyle\indenti{\Delta V_k^x} = -(\lambda_{min}(Q) - \frac{\|B^\top PB\|_2r_{max}^2}{\|x_{pk}\|^2} && \nonumber \\
    &\textstyle\indenti{\Delta V_k^x = -.} - \frac{\|B^\top PB\|_2\|\varepsilon_{k+1}\|^2}{\|x_{pk}\|^2} - \frac{2\|A_m^\top PB\|_2r_{max}}{\|x_{pk}\|} && \nonumber \\
    &\textstyle\indenti{\Delta V_k^x = -.} - \frac{2\|A_m^\top PB\|_2\|\varepsilon_{k+1}\|}{\|x_{pk}\|} && \nonumber \\
    &\textstyle\indenti{\Delta V_k^x = -.} - \frac{2\|B^\top PB\|_2r_{max}\|\varepsilon_{k+1}\|}{\|x_{pk}\|^2})\|x_{pk}\|^2 && \nonumber \\
    &\textstyle\indenti{\Delta V_k^x} \leq -(\lambda_{min}(Q) - \frac{2\|A_m^\top PB\|_2(r_{max} + \|\varepsilon_{k+1}\|)}{\|x_{pk}\|} && \nonumber \\
    &\textstyle\indenti{\Delta V_k^x = -.} - \frac{2\|B^\top PB\|_2(r_{max}^2 + \|\varepsilon_{k+1}\|^2)}{\|x_{pk}\|^2})\frac{1}{\lambda_{max}(P)}V_k^x \label{eqn:V_x_increment} &&
\end{flalign}
whenever $\|x_{pk}\| \neq 0$. Choose any constant $c_1$ such that $0 < c_1 < \lambda_{min}(Q)$. Then, $\Delta V_k^x \leq -\frac{c_1}{\lambda_{max}(P)}V_k^x$ whenever
\begin{gather}
    \begin{aligned}
        &\textstyle\lambda_{min}(Q) - \frac{2\|A_m^\top PB\|_2(r_{max} + \|\varepsilon_{k+1}\|)}{\|x_{pk}\|} \nonumber \\
        &\textstyle- \frac{2\|B^\top PB\|_2(r_{max}^2 + \|\varepsilon_{k+1}\|^2)}{\|x_{pk}\|^2} \geq c_1 \iff
    \end{aligned} \\
    \begin{aligned}
        &\textstyle\lambda_{min}(Q) - \frac{2\|A_m^\top PB\|_2\|\varepsilon_{k+1}\|}{\|x_{pk}\|} - \frac{2\|B^\top PB\|_2\|\varepsilon_{k+1}\|^2}{\|x_{pk}\|^2} \nonumber \\
        &\textstyle- \frac{2\|A_m^\top PB\|_2r_{max}}{\|x_{pk}\|} - \frac{2\|B^\top PB\|_2r_{max}^2}{\|x_{pk}\|^2} \geq c_1 \iff
    \end{aligned} \\
    \begin{aligned}
        &\textstyle\frac{1}{\|x_{pk}\|} \leq \sqrt{\frac{\|A_m^\top PB\|_2^2}{4\|B^\top PB\|_2^2r_{max}^2} + \frac{d_1\left(\frac{\|\varepsilon_{k+1}\|}{\|x_{pk}\|}\right) - c_1}{2\|B^\top PB\|_2r_{max}^2}} \\
        &\textstyle\indenti{\frac{1}{\|x_{pk}\|} =} - \frac{\|A_m^\top PB\|_2}{2\|B^\top PB\|_2r_{max}} \label{eqn:x_mag_quadratic}
    \end{aligned}
\end{gather}
where
\begin{equation}
    \begin{aligned}
        \textstyle d_1(\frac{\|\varepsilon_{k+1}\|}{\|x_{pk}\|}) =&\textstyle \lambda_{min}(Q) - \frac{2\|A_m^\top PB\|_2\|\varepsilon_{k+1}\|}{\|x_{pk}\|} \\
        &\textstyle- \frac{2\|B^\top PB\|_2\|\varepsilon_{k+1}\|^2}{\|x_{pk}\|^2}. \label{eqn:d_1}
    \end{aligned}
\end{equation}
Further define
\begin{equation}
    \begin{aligned}
        \textstyle d_2(c_2) =& \textstyle\bigg(\sqrt{\frac{\|A_m^\top PB\|_2^2}{4\|B^\top PB\|_2^2r_{max}^2} + \frac{c_2}{2\|B^\top PB\|_2r_{max}^2}} \\
        &\textstyle- \frac{\|A_m^\top PB\|_2}{2\|B^\top PB\|_2r_{max}}\bigg)^{-1}. \label{eqn:d_2}
    \end{aligned}
\end{equation}

The proof concludes as follows: choose any constants $c_1 \in (0, \lambda_{min}(Q))$ and $c_2 \in (0, \lambda_{min}(Q) - c_1)$. Then, because $\lim_{k \to \infty} \frac{\|\varepsilon_{k+1}\|}{\|x_{pk}\|} = 0$, there exists a time step $K(c_1, c_2)$ such that $d_1(\frac{\|\varepsilon_{k+1}\|}{\|x_{pk}\|}) - c_1 \geq c_2\ \forall k \geq K(c_1, c_2)$. Thus, for all $k \geq K(c_1, c_2)$, we have 
\begin{gather*}
    \textstyle V_k^x \geq \lambda_{min}(P)d_2(c_2)^2 \implies \|x_{pk}\| \geq d_2(c_2) \iff \\
    \textstyle\frac{1}{\|x_{pk}\|} \leq \frac{1}{d_2(c_2)} \implies \eqref{eqn:x_mag_quadratic} \implies \Delta V_k^x \leq -\frac{c_1}{\lambda_{max}(P)}V_k^x.
\end{gather*}

Therefore, for all but finitely many time steps, $V_k^x$ converges exponentially to a compact set, implying that $x_{pk}$ and therefore $\phi_k$ is bounded.

\subsection{Proof of Proposition \ref{pro:ht_Lyapunov}} \label{subsec:ht_Lyapunov_proof}


Consider the candidate Lyapunov function given in \eqref{eqn:ht_Lyapunov_function}. Defining $\tilde{\Xi}_k = \hat{\Xi}_k - \Theta_*$
and applying \eqref{eqn:ht_Xi}-\eqref{eqn:ht_L_gradient_k+1}, we obtain
\begin{flalign*}
    &\Delta V_k = \|\tilde{\Xi}_{k+1}\|_F^2 - \|\tilde{\Xi}_k\|_F^2 && \\
    &\indenti{\Delta V_k =} + \|\hat{\Theta}_{k+1} - \hat{\Xi}_{k+1}\|_F^2 - \|\hat{\Theta}_k - \hat{\Xi}_k\|_F^2 &&
\end{flalign*}
\begin{flalign*}
    &\indenti{\Delta V_k} = \|\tilde{\Xi}_k - \frac{\gamma}{\N_k}\nabla L_k(\hat{\Theta}_{k+1})\|_F^2 - \|\tilde{\Xi}_k\|_F^2 && \\
    &\indenti{\Delta V_k =} + \|\bar{\Theta}_k - \beta(\bar{\Theta}_k - \hat{\Xi}_k) - \hat{\Xi}_k + \frac{\gamma}{\N_k}\nabla L_k(\hat{\Theta}_{k+1})\|_F^2 && \\
    &\indenti{\Delta V_k =} - \|\hat{\Theta}_k - \hat{\Xi}_k\|_F^2 &&
\end{flalign*}
\begin{flalign*}
    &\indenti{\Delta V_k} = \frac{\gamma^2}{\N_k^2}\|\nabla L_k(\hat{\Theta}_{k+1})\|_F^2 - \frac{2\gamma}{\N_k}\Tr[\tilde{\Xi}_k^\top\nabla L_k(\hat{\Theta}_{k+1})] && \\
    &\indenti{\Delta V_k =} + (1 - \beta)^2\|\bar{\Theta}_k - \hat{\Xi}_k\|_F^2 - \|\hat{\Theta}_k - \hat{\Xi}_k\|_F^2 && \\
    &\indenti{\Delta V_k =} + \frac{2\gamma(1 - \beta)}{\N_k}\Tr[(\bar{\Theta}_k - \hat{\Xi}_k)^\top\nabla L_k(\hat{\Theta}_{k+1})] && \\
    &\indenti{\Delta V_k =} + \frac{\gamma^2}{\N_k^2}\|\nabla L_k(\hat{\Theta}_{k+1})\|_F^2 &&
\end{flalign*}
\begin{flalign*}
    &\indenti{\Delta V_k} = \frac{2\gamma^2}{\N_k^2}\|\nabla L_k(\hat{\Theta}_{k+1})\|_F^2 - \frac{2\gamma}{\N_k}\Tr[\tilde{\Theta}_{k+1}^\top\nabla L_k(\hat{\Theta}_{k+1})] && \\
    &\indenti{\Delta V_k =} - \frac{2\gamma}{\N_k}\Tr[(\hat{\Xi}_k - \hat{\Theta}_{k+1})^\top\nabla L_k(\hat{\Theta}_{k+1})] && \\
    &\indenti{\Delta V_k =} + \|\bar{\Theta}_k - \hat{\Xi}_k\|_F^2 - \|\hat{\Theta}_k - \hat{\Xi}_k\|_F^2 && \\
    &\indenti{\Delta V_k =} - \beta(2 - \beta)\|\bar{\Theta}_k - \hat{\Xi}_k\|_F^2 && \\
    &\indenti{\Delta V_k =} + \frac{2\gamma(1 - \beta)}{\N_k}\Tr[(\bar{\Theta}_k - \hat{\Xi}_k)^\top\nabla L_k(\hat{\Theta}_{k+1})] &&
\end{flalign*}
\begin{flalign*}
    &\indenti{\Delta V_k} = \frac{2\gamma^2\|\phi_k\|^2}{\N_k^2}\|\tilde{\Theta}_{k+1}\phi_k\|^2 - \frac{2\gamma}{\N_k}\|\tilde{\Theta}_{k+1}\phi_k\|^2 && \\
    &\indenti{\Delta V_k =} + \|\bar{\Theta}_k - \hat{\Xi}_k\|_F^2 - \|\hat{\Theta}_k - \hat{\Xi}_k\|_F^2 && \\
    &\indenti{\Delta V_k =} - \beta(2 - \beta)\|\bar{\Theta}_k - \hat{\Xi}_k\|_F^2 && \\
    &\indenti{\Delta V_k =} + \frac{2\gamma(1 - \beta)}{\N_k}\Tr[(\bar{\Theta}_k - \hat{\Xi}_k)^\top\nabla L_k(\hat{\Theta}_{k+1})] && \\
    &\indenti{\Delta V_k =} + \frac{2\gamma}{\N_k}\Tr[(\hat{\Theta}_{k+1} - \hat{\Xi}_k)^\top\nabla L_k(\hat{\Theta}_{k+1})] &&
\end{flalign*}
\begin{flalign*}
    &\indenti{\Delta V_k} = -\frac{2\gamma}{\N_k}\Big(1 - \frac{\gamma\|\phi_k\|^2}{\N_k}\Big)\|\tilde{\Theta}_{k+1}\phi_k\|^2 && \\
    &\indenti{\Delta V_k =} + \|\hat{\Theta}_k - \frac{\gamma\beta}{\N_k}\nabla L_k(\hat{\Theta}_k) - \hat{\Xi}_k\|_F^2 - \|\hat{\Theta}_k - \hat{\Xi}_k\|_F^2 && \\
    &\indenti{\Delta V_k =} - \beta(2 - \beta)\|\bar{\Theta}_k - \hat{\Xi}_k\|_F^2 && \\
    &\indenti{\Delta V_k =} + \frac{4\gamma(1 - \beta)}{\N_k}\Tr[(\bar{\Theta}_k - \hat{\Xi}_k)^\top\nabla L_k(\hat{\Theta}_{k+1})] &&
\end{flalign*}
\begin{flalign*}
    &\indenti{\Delta V_k} = -\frac{2\gamma}{\N_k}\Big(1 - \frac{\gamma\|\phi_k\|^2}{\N_k}\Big)\|\tilde{\Theta}_{k+1}\phi_k\|^2 && \\
    &\indenti{\Delta V_k =} + \frac{\gamma^2\beta^2}{\N_k^2}\|\nabla L_k(\hat{\Theta}_k)\|_F^2 && \\
    &\indenti{\Delta V_k =} - \frac{2\gamma\beta}{\N_k}\Tr[(\hat{\Theta}_k - \hat{\Xi}_k)^\top\nabla L_k(\hat{\Theta}_k)] && \\
    &\indenti{\Delta V_k =} - \beta(2 - \beta)\|\bar{\Theta}_k - \hat{\Xi}_k\|_F^2 && \\
    &\indenti{\Delta V_k =} + \frac{4\gamma(1 - \beta)}{\N_k}\Tr[(\bar{\Theta}_k - \hat{\Xi}_k)^\top\nabla L_k(\hat{\Theta}_{k+1})] &&
\end{flalign*}
\begin{flalign*}
    &\indenti{\Delta V_k} = -\frac{2\gamma}{\N_k}\Big(1 - \frac{\gamma\|\phi_k\|^2}{\N_k}\Big)\|\tilde{\Theta}_{k+1}\phi_k\|^2 && \\
    &\indenti{\Delta V_k =} - \frac{\gamma^2\beta^2}{\N_k^2}\|\nabla L_k(\hat{\Theta}_k)\|_F^2 && \\
    &\indenti{\Delta V_k =} - \frac{2\gamma\beta}{\N_k}\Tr[(\bar{\Theta}_k - \hat{\Xi}_k)^\top\nabla L_k(\hat{\Theta}_k)] && \\
    &\indenti{\Delta V_k =} - \beta(2 - \beta)\|\bar{\Theta}_k - \hat{\Xi}_k\|_F^2 && \\
    &\indenti{\Delta V_k =} + \frac{4\gamma(1 - \beta)}{\N_k}\Tr[(\bar{\Theta}_k - \hat{\Xi}_k)^\top\nabla L_k(\hat{\Theta}_{k+1})] &&
\end{flalign*}
\begin{flalign*}
    &\indenti{\Delta V_k} = -\frac{2\gamma}{\N_k}(1 - \frac{\gamma\|\phi_k\|^2}{\N_k})\|\tilde{\Theta}_{k+1}\phi_k\|^2 && \\
    &\indenti{\Delta V_k =} - \frac{\gamma^2\beta^2}{\N_k^2}\|\nabla L_k(\hat{\Theta}_k)\|_F^2 && \\
    &\indenti{\Delta V_k =} - \frac{2\gamma\beta}{\N_k}\phi_k^\top(\bar{\Theta}_k - \hat{\Xi}_k)^\top(\bar{\Theta}_k - \beta(\bar{\Theta}_k - \hat{\Xi}_k) - \Theta_* && \\
    &\indenti{\Delta V_k = \hspace{20pt}} + \hat{\Theta}_k - \bar{\Theta}_k + \beta(\bar{\Theta}_k - \hat{\Xi}_k))\phi_k && \\
    &\indenti{\Delta V_k =} - \beta(2 - \beta)\|\bar{\Theta}_k - \hat{\Xi}_k\|_F^2 && \\
    &\indenti{\Delta V_k =} + \frac{4\gamma(1 - \beta)}{\N_k}\phi_k^\top(\bar{\Theta}_k - \hat{\Xi}_k)^\top\tilde{\Theta}_{k+1}\phi_k &&
\end{flalign*}
\begin{flalign*}
    &\indenti{\Delta V_k} = -\frac{2\gamma}{\N_k}(1 - \frac{\gamma\|\phi_k\|^2}{\N_k})\|\tilde{\Theta}_{k+1}\phi_k\|^2 && \\
    &\indenti{\Delta V_k =} - \frac{\gamma^2\beta^2}{\N_k^2}\|\nabla L_k(\hat{\Theta}_k)\|_F^2 && \\
    &\indenti{\Delta V_k =} - \frac{2\gamma^2\beta^2}{\N_k^2}\phi_k^\top(\bar{\Theta}_k - \hat{\Xi}_k)^\top\nabla L_k(\hat{\Theta}_k)\phi_k && \\
    &\indenti{\Delta V_k =} - \frac{2\gamma\beta^2}{\N_k}\|(\bar{\Theta}_k - \hat{\Xi}_k)\phi_k\|^2 && \\
    &\indenti{\Delta V_k =} - \beta(2 - \beta)\|\bar{\Theta}_k - \hat{\Xi}_k\|_F^2 && \\
    &\indenti{\Delta V_k =} + \frac{2\gamma(2 - 3\beta)}{\N_k}\phi_k^\top(\bar{\Theta}_k - \hat{\Xi}_k)^\top\tilde{\Theta}_{k+1}\phi_k &&
\end{flalign*}
\begin{flalign*}
    &\indenti{\Delta V_k} \leq -\frac{2\gamma}{\N_k}(1 - \frac{\gamma\|\phi_k\|^2}{\N_k})\|\tilde{\Theta}_{k+1}\phi_k\|^2 && \\
    &\indenti{\Delta V_k =} - \frac{\gamma^2\beta^2}{\N_k^2}\|\nabla L_k(\hat{\Theta}_k)\|_F^2 && \\
    &\indenti{\Delta V_k =} - \frac{2\gamma^2\beta^2}{\N_k^2}\phi_k^\top(\bar{\Theta}_k - \hat{\Xi}_k)^\top\nabla L_k(\hat{\Theta}_k)\phi_k && \\
    &\indenti{\Delta V_k =} - \frac{2\gamma\beta^2}{\N_k}\|(\bar{\Theta}_k - \hat{\Xi}_k)\phi_k\|^2 && \\
    &\indenti{\Delta V_k =} - \frac{\beta(2 - \beta)\|\phi_k\|^2}{\N_k}\|\bar{\Theta}_k - \hat{\Xi}_k\|_F^2 && \\
    &\indenti{\Delta V_k =} + \frac{2\gamma(2 - 3\beta)}{\N_k}\phi_k^\top(\bar{\Theta}_k - \hat{\Xi}_k)^\top\tilde{\Theta}_{k+1}\phi_k &&
\end{flalign*}
\begin{flalign*}
    &\indenti{\Delta V_k} \leq -\frac{2\gamma}{\N_k}(1 - \frac{\gamma\|\phi_k\|^2}{\N_k})\|\tilde{\Theta}_{k+1}\phi_k\|^2 && \\
    &\indenti{\Delta V_k =} - \frac{\gamma^2\beta^2}{\N_k^2}\|\nabla L_k(\hat{\Theta}_k)\|_F^2 && \\
    &\indenti{\Delta V_k =} + \frac{2\gamma^2\beta^2\|\phi_k\|^2}{\N_k^2}\|\bar{\Theta}_k - \hat{\Xi}_k\|_F\|\nabla L_k(\hat{\Theta}_k)\|_F && \\
    &\indenti{\Delta V_k =} - \frac{2\gamma\beta^2}{\N_k}\|(\bar{\Theta}_k - \hat{\Xi}_k)\phi_k\|^2 && \\
    &\indenti{\Delta V_k =} - \frac{\beta(2 - \beta)\|\phi_k\|^2}{\N_k}\|\bar{\Theta}_k - \hat{\Xi}_k\|_F^2 && \\
    &\indenti{\Delta V_k =} + \frac{2\gamma|2 - 3\beta|\|\phi_k\|}{\N_k}\|(\bar{\Theta}_k - \hat{\Xi}_k)\|_F\|\tilde{\Theta}_{k+1}\phi_k\| &&
\end{flalign*}
\begin{flalign*}
    &\indenti{\Delta V_k} \leq -\frac{2\gamma}{\N_k}(1 - \frac{\gamma\|\phi_k\|^2}{\N_k})\|\tilde{\Theta}_{k+1}\phi_k\|^2 && \\
    &\indenti{\Delta V_k =} - \frac{\gamma^2\beta^2}{\N_k^2}\|\nabla L_k(\hat{\Theta}_k)\|_F^2 && \\
    &\indenti{\Delta V_k =} + \frac{2\gamma^2\beta^2\|\phi_k\|}{\N_k^{3/2}}\|\bar{\Theta}_k - \hat{\Xi}_k\|_F\|\nabla L_k(\hat{\Theta}_k)\|_F && \\
    &\indenti{\Delta V_k =} - \frac{2\gamma\beta^2}{\N_k}\|(\bar{\Theta}_k - \hat{\Xi}_k)\phi_k\|^2 && \\
    &\indenti{\Delta V_k =} - \frac{\beta(2 - \beta)\|\phi_k\|^2}{\N_k}\|\bar{\Theta}_k - \hat{\Xi}_k\|_F^2 && \\
    &\indenti{\Delta V_k =} + \frac{2\gamma|2 - 3\beta|\|\phi_k\|}{\N_k}\|\bar{\Theta}_k - \hat{\Xi}_k\|_F\|\tilde{\Theta}_{k+1}\phi_k\|. &&
\end{flalign*}
where the above inequalities use the fact that $\N_k \geq \|\phi_k\|^2$ and require $0 < \beta < 2$ and $\gamma > 0$. We now complete squares with the $\|\nabla L_k(\hat{\Theta}_k)\|_F$ terms and the $\|\phi_k\|^2\|\bar{\Theta}_k - \hat{\Xi}_k\|_F^2$ term to obtain
\begin{flalign*}
    &\Delta V_k \leq -\frac{2\gamma}{\N_k}(1 - \frac{\gamma\|\phi_k\|^2}{\N_k})\|\tilde{\Theta}_{k+1}\phi_k\|^2 && \\
    &\indenti{\Delta V_k =} - \frac{\gamma^2\beta^2}{\N_k}\Big(\frac{1}{\sqrt{\N_k}}\|\nabla L_k(\hat{\Theta}_k)\|_F - \|\phi_k\|\|\bar{\Theta}_k - \hat{\Xi}_k\|_F)^2 && \\
    &\indenti{\Delta V_k =} - \frac{2\gamma\beta^2}{\N_k}\|(\bar{\Theta}_k - \hat{\Xi}_k)\phi_k\|^2 && \\
    &\indenti{\Delta V_k =} - \frac{\beta(2 - (1 + \gamma^2)\beta)\|\phi_k\|^2}{\N_k}\|\bar{\Theta}_k - \hat{\Xi}_k\|_F^2 && \\
    &\indenti{\Delta V_k =} + \frac{2\gamma|2 - 3\beta|\|\phi_k\|}{\N_k}\|\bar{\Theta}_k - \hat{\Xi}_k\|_F\|\tilde{\Theta}_{k+1}\phi_k\| &&
\end{flalign*}

We then complete squares with the $\tilde{\Theta}_{k+1}^\top\phi_k$ terms and the $\|\phi_k\|^2\|\bar{\Theta}_k - \hat{\Xi}_k\|^2$ term and note that $1 - \gamma \leq 1 - \frac{\gamma\|\phi_k\|^2}{\N_k} \leq 1$ to obtain
\begin{flalign*}
    &\Delta V_k \leq -\frac{\gamma}{\N_k}\Big(2(1 - \gamma) - \frac{\gamma(2 - 3\beta)^2}{\beta(2 - (1 + \gamma^2)\beta)}\Big)\|\tilde{\Theta}_{k+1}\phi_k\|^2 && \\
    &\indenti{\Delta V_k =} - \frac{\gamma^2\beta^2}{\N_k}\Big(\frac{1}{\sqrt{\N_k}}\|\nabla L_k(\hat{\Theta}_k)\|_F - \|\phi_k\|\|\bar{\Theta}_k - \hat{\Xi}_k\|_F\Big)^2 && \\
    &\indenti{\Delta V_k =} - \frac{2\gamma\beta^2}{\N_k}\|(\bar{\Theta}_k - \hat{\Xi}_k)\phi_k\|^2 && \\
    &\indenti{\Delta V_k =} - \frac{\beta(2 - (1 + \gamma^2)\beta)}{\N_k}\Big(\frac{\gamma|2 - 3\beta|}{\beta(2 - (1 + \gamma^2)\beta)}\|\tilde{\Theta}_{k+1}\phi_k\| && \\
    &\indenti{\Delta V_k = - \frac{\beta(2 - (1 + \gamma^2)\beta)}{\N_k}..} - \|\phi_k\|\|\bar{\Theta}_k - \hat{\Xi}_k\|_F\Big)^2 && \\
    &\indenti{\Delta V_k} \leq 0 &&
\end{flalign*}
if $0 < \beta < 2$, $\gamma > 0$, $2 - (1 + \gamma^2)\beta > 0$, and $2(1 - \gamma) - \frac{\gamma(2 - 3\beta)^2}{\beta(2 - (1 + \gamma^2)\beta)} > 0$. 

This is sufficient to prove that $V_k$ is a Lyapunov function. However, in order to prove Theorem \ref{thm:ht_error_bounded}, we need the increment in terms of $\varepsilon_{k+1}$. Rearranging \eqref{eqn:ht_Theta_bar}-\eqref{eqn:ht_Theta} and using \eqref{eqn:error_model_1}, we obtain
\begin{flalign}
    &\tilde{\Theta}_{k+1}\phi_k = (\bar{\Theta}_k - \beta(\bar{\Theta}_k - \hat{\Xi}_k) - \Theta_*)\phi_k && \nonumber \\
    &\indenti{\tilde{\Theta}_{k+1}\phi_k} = \Big(1 - \frac{\gamma\beta\|\phi_k\|^2}{\N_k}\Big)\tilde{\Theta}_k\phi_k - \beta(\bar{\Theta}_k - \hat{\Xi}_k)\phi_k && \nonumber \\
    &\indenti{\tilde{\Theta}_{k+1}\phi_k} = \Big(1 - \frac{\gamma\beta\|\phi_k\|^2}{\N_k}\Big)\varepsilon_{k+1} - \beta(\bar{\Theta}_k - \hat{\Xi}_k)\phi_k, \label{eqn:ht_a_posteriori_prediction_error} &&
\end{flalign}
\begin{flalign}
    &\|\tilde{\Theta}_{k+1}\phi_k\|^2 = \Big(1 - \frac{\gamma\beta\|\phi_k\|^2}{\N_k}\Big)^2\|\varepsilon_{k+1}\|^2 && \nonumber \\
    &\indenti{\|\tilde{\Theta}_{k+1}\phi_k\|^2 =} - 2\beta\Big(1 - \frac{\gamma\beta\|\phi_k\|^2}{\N_k}\Big)\phi_k^\top(\bar{\Theta}_k - \hat{\Xi}_k)^\top\varepsilon_{k+1} && \nonumber \\
    &\indenti{\|\tilde{\Theta}_{k+1}\phi_k\|^2 =} + \beta^2\|(\bar{\Theta}_k - \hat{\Xi}_k)\phi_k\|^2 \label{eqn:ht_a_posteriori_error_sqr} &&
\end{flalign}
Finally, defining $\alpha$ as
\begin{equation} \label{eqn:ht_alpha}
    \alpha = 2(1 - \gamma) - \frac{\gamma(2 - 3\beta)^2}{\beta(2 - (1 + \gamma^2)\beta)} > 0,
\end{equation}
we substitute the expression above into the Lyapunov increment and rearrange to yield
\begin{flalign*}
    &\Delta V_k \leq -\frac{\gamma\alpha}{\N_k}\|\tilde{\Theta}_{k+1}\phi_k\|^2 && \\
    &\indenti{\Delta V_k =} - \frac{2\gamma\beta^2}{\N_k}\|(\bar{\Theta}_k - \hat{\Xi}_k)\phi_k\|^2 &&
\end{flalign*}
\begin{flalign*}
    &\indenti{\Delta V_k} = -\frac{\gamma\alpha}{\N_k}\Big(1 - \frac{\gamma\beta\|\phi_k\|^2}{\N_k}\Big)^2\|\varepsilon_{k+1}\|^2 && \\
    &\indenti{\Delta V_k =} + \frac{2\gamma\beta\alpha}{\N_k}\Big(1 - \frac{\gamma\beta\|\phi_k\|^2}{\N_k}\Big)\phi_k^\top(\bar{\Theta}_k - \hat{\Xi}_k)^\top\varepsilon_{k+1} && \\
    &\indenti{\Delta V_k =} - \frac{\gamma\beta^2(2 + \alpha)}{\N_k}\|(\bar{\Theta}_k - \hat{\Xi}_k)\phi_k\|^2 &&
\end{flalign*}
Completing the square with the $\varepsilon_{k+1}$ and $(\bar{\Theta}_k - \hat{\Xi}_k)^\top\phi_k$ terms and noting that $1 - \gamma\beta < 1 - \frac{\gamma\beta\|\phi_k\|^2}{\N_k} \leq 1$, we obtain the final expression:
\begin{flalign*}
    &\Delta V_k \leq -\frac{\gamma\alpha(1 - \gamma\beta)^2}{\N_k}\Big(1 - \frac{\alpha}{2 + \alpha}\Big)\|\varepsilon_{k+1}\|^2 && \\
    &\indenti{\Delta V_k =} - \frac{\gamma(2 + \alpha)}{\N_k}\Big(\frac{\alpha}{2 + \alpha}\Big(1 - \frac{\gamma\beta\|\phi_k\|^2}{\N_k}\Big)\varepsilon_{k+1} && \\
    &\indenti{\Delta V_k = - \frac{\gamma(2 + \alpha)}{\N_k}..} - \beta(\bar{\Theta}_k - \hat{\Xi}_k)^\top\phi_k\Big)^2 && \\
    &\indenti{\Delta V_k} \leq 0 &&
\end{flalign*}
since $\alpha > 0 \implies 1 - \frac{\alpha}{2 + \alpha} > 0$, and the restrictions on $\gamma$ and $\beta$ imply that $\gamma\beta < 1$.

In summary, we have
\begin{equation} \label{eqn:ht_Lyapunov_increment}
    \Delta V_k \leq -\gamma\alpha(1 - \gamma\beta)^2\Big(1 - \frac{\alpha}{2 + \alpha}\Big)\frac{\|\varepsilon_{k+1}\|^2}{\N_k} \leq 0.
\end{equation}
\subsection{Proof of Theorem \ref{thm:ht_error_bounded}} \label{subsec:ht_stability_proof}


    Consider the Lyapunov function in \eqref{eqn:ht_Lyapunov_function}. From Proposition \ref{pro:ht_Lyapunov}, we know that $V_k \geq 0$ and $\Delta V_k \leq 0$. It follows immediately that
    \begin{gather}
        0 \leq \lim_{k \tends \infty} V_k \leq V_0 \implies \nonumber \\
        0 \leq V_0 + \sum_{k = 0}^\infty \Delta V_k \leq V_0 \implies \nonumber \\
        -V_0 \leq \sum_{k = 0}^\infty \Delta V_k \leq 0 \implies \nonumber \\
        \lim_{k \tends \infty} \Delta V_k = 0.
    \end{gather}
    Substituting \eqref{eqn:ht_Lyapunov_increment}, we get
    \begin{equation} \label{eqn:ht_lim_goes_to_0}
        \lim_{k \tends \infty} \frac{\|\varepsilon_{k+1}\|^2}{\N_k} = 0.
    \end{equation}
    
    The remainder of the proof proceeds identically to that of Theorem \ref{thm:gd_error_bounded}.
    Equation \eqref{eqn:ht_lim_goes_to_0} implies that either (1) $\lim_{k \to \infty} \|\varepsilon_k\| = 0$ or (2) $\lim_{k \to \infty} \frac{\|\varepsilon_{k+1}\|}{\|x_{pk}\|} = 0$. If case (1) holds, then the proof is complete. Under case (2), the proof is completed using the same Lyapunov analysis as found in Subsection \ref{subsec:gd_stability_proof}.
\subsection{A LARGER REGION OF ALLOWABLE HIGH-ORDER TUNER GAINS} \label{subsec:ht_hyperparameters}

In this section, we rewrite the increment of the Lyapunov function in \eqref{eqn:ht_Lyapunov_function} in a different manner than in Section \ref{subsec:ht_Lyapunov_proof} in order to graphically show a larger range of allowable hyperparameters $\gamma$ and $\beta$. From Section \ref{subsec:ht_Lyapunov_proof}, we have
\begin{flalign*}
    &\Delta V_k = -\frac{2\gamma}{\N_k}\Big(1 - \frac{\gamma\|\phi_k\|^2}{\N_k}\Big)\|\tilde{\Theta}_{k+1}\phi_k\|^2 && \\
    &\indenti{\Delta V_k =} - \frac{\gamma^2\beta^2}{\N_k^2}\|\nabla L_k(\hat{\Theta}_k)\|_F^2 && \\
    &\indenti{\Delta V_k =} - \frac{2\gamma\beta}{\N_k}\Tr[(\bar{\Theta}_k - \hat{\Xi}_k)^\top\nabla L_k(\hat{\Theta}_k)] && \\
    &\indenti{\Delta V_k =} - \beta(2 - \beta)\|\bar{\Theta}_k - \hat{\Xi}_k\|_F^2 && \\
    &\indenti{\Delta V_k =} + \frac{4\gamma(1 - \beta)}{\N_k}\Tr[(\bar{\Theta}_k - \hat{\Xi}_k)^\top\nabla L_k(\hat{\Theta}_{k+1})] &&
\end{flalign*}
\begin{flalign*}
    &\indenti{\Delta V_k} = -\frac{2\gamma}{\N_k}\Big(1 - \frac{\gamma\|\phi_k\|^2}{\N_k}\Big)\|\tilde{\Theta}_{k+1}\phi_k\|^2 && \\
    &\indenti{\Delta V_k =} - \frac{\gamma^2\beta^2\|\phi_k\|^2}{\N_k^2}\|\varepsilon_{k+1}\|^2 && \\
    &\indenti{\Delta V_k =} - \frac{2\gamma\beta}{\N_k}\phi_k^\top(\bar{\Theta}_k - \hat{\Xi}_k)^\top\varepsilon_{k+1} && \\
    &\indenti{\Delta V_k =} - \beta(2 - \beta)\|\bar{\Theta}_k - \hat{\Xi}_k\|_F^2 && \\
    &\indenti{\Delta V_k =} + \frac{4\gamma(1 - \beta)}{\N_k}\phi_k^\top(\bar{\Theta}_k - \hat{\Xi}_k)^\top\tilde{\Theta}_{k+1}\phi_k &&
\end{flalign*}
Define
\begin{equation} \label{eqn:ht_lambda}
    \lambda_k = \frac{\|\phi_k\|^2}{\N_k} \in [0, 1].
\end{equation}
Now, we apply \eqref{eqn:ht_a_posteriori_prediction_error}, \eqref{eqn:ht_a_posteriori_error_sqr}, and \eqref{eqn:ht_lambda} to obtain
\begin{flalign*}
    &\Delta V_k = -\frac{2\gamma}{\N_k}(1 - \gamma\lambda_k)(1 - \gamma\beta\lambda_k)^2\|\varepsilon_{k+1}\|^2 && \\
    &\indenti{\Delta V_k =} + \frac{4\gamma\beta}{\N_k}(1 - \gamma\lambda_k)(1 - \gamma\beta\lambda_k)\phi_k^\top(\bar{\Theta}_k - \hat{\Xi}_k)^\top\varepsilon_{k+1} && \\
    &\indenti{\Delta V_k =} - \frac{2\gamma\beta^2}{\N_k}(1 - \gamma\lambda_k)\|(\bar{\Theta}_k - \hat{\Xi}_k)\phi_k\|^2 && \\
    &\indenti{\Delta V_k =} - \frac{\gamma^2\beta^2\lambda_k}{\N_k}\|\varepsilon_{k+1}\|^2 && \\
    &\indenti{\Delta V_k =} - \frac{2\gamma\beta}{\N_k}\phi_k^\top(\bar{\Theta}_k - \hat{\Xi}_k)^\top\varepsilon_{k+1} && \\
    &\indenti{\Delta V_k =} - \beta(2 - \beta)\|\bar{\Theta}_k - \hat{\Xi}_k\|_F^2 && \\
    &\indenti{\Delta V_k =} + \frac{4\gamma(1 - \beta)}{\N_k}(1 - \gamma\beta\lambda_k)\phi_k^\top(\bar{\Theta}_k - \hat{\Xi}_k)^\top\varepsilon_{k+1} && \\
    &\indenti{\Delta V_k =} - \frac{4\gamma\beta(1 - \beta)}{\N_k}\|(\bar{\Theta}_k - \hat{\Xi}_k)\phi_k\|^2 &&
\end{flalign*}
\begin{flalign*}
    &\indenti{\Delta V_k} \leq -\frac{2\gamma}{\N_k}(1 - \gamma\lambda_k)(1 - \gamma\beta\lambda_k)^2\|\varepsilon_{k+1}\|^2 && \\
    &\indenti{\Delta V_k =} + \frac{4\gamma\beta}{\N_k}(1 - \gamma\lambda_k)(1 - \gamma\beta\lambda_k)\phi_k^\top(\bar{\Theta}_k - \hat{\Xi}_k)^\top\varepsilon_{k+1} && \\
    &\indenti{\Delta V_k =} - \frac{2\gamma\beta^2}{\N_k}(1 - \gamma\lambda_k)\|(\bar{\Theta}_k - \hat{\Xi}_k)\phi_k\|^2 && \\
    &\indenti{\Delta V_k =} - \frac{\gamma^2\beta^2\lambda_k}{\N_k}\|\varepsilon_{k+1}\|^2 && \\
    &\indenti{\Delta V_k =} - \frac{2\gamma\beta}{\N_k}\phi_k^\top(\bar{\Theta}_k - \hat{\Xi}_k)^\top\varepsilon_{k+1} && \\
    &\indenti{\Delta V_k =} - \frac{\beta(2 - \beta)}{\N_k}\|(\bar{\Theta}_k - \hat{\Xi}_k)\phi_k\|^2 && \\
    &\indenti{\Delta V_k =} + \frac{4\gamma(1 - \beta)}{\N_k}(1 - \gamma\beta\lambda_k)\phi_k^\top(\bar{\Theta}_k - \hat{\Xi}_k)^\top\varepsilon_{k+1} && \\
    &\indenti{\Delta V_k =} - \frac{4\gamma\beta(1 - \beta)}{\N_k}\|(\bar{\Theta}_k - \hat{\Xi}_k)\phi_k\|^2 &&
\end{flalign*}
\begin{flalign*}
    &\indenti{\Delta V_k} = -\frac{\gamma}{\N_k}(2(1 - \gamma\lambda_k)(1 - \gamma\beta\lambda_k)^2 + \gamma\beta^2\lambda_k)\|\varepsilon_{k+1}\|^2 && \\ 
    &\indenti{\Delta V_k =} + \frac{2\gamma}{\N_k}(2\beta(1 - \gamma\lambda_k)(1 - \gamma\beta\lambda_k) - \beta && \\
    &\indenti{\Delta V_k = + \frac{2\gamma}{\N_k}.} + 2(1 - \beta)(1 - \gamma\beta\lambda_k))\phi_k^\top(\bar{\Theta}_k - \hat{\Xi}_k)^\top\varepsilon_{k+1} && \\
    &\indenti{\Delta V_k =} - \frac{\beta}{\N_k}(2\gamma\beta(1 - \gamma\lambda_k) + 2 - \beta && \\
    &\indenti{\Delta V_k = - \frac{\beta}{\N_k}.} + 4\gamma(1 - \beta))\|(\bar{\Theta}_k - \hat{\Xi}_k)\phi_k\|^2 &&
\end{flalign*}
\begin{flalign*}
    &\indenti{\Delta V_k} = -\frac{a(\gamma, \beta, \lambda_k)}{\N_k}\|\varepsilon_{k+1}\|^2 && \\ 
    &\indenti{\Delta V_k =} + \frac{2b(\gamma, \beta, \lambda_k)}{\N_k}\phi_k^\top(\bar{\Theta}_k - \hat{\Xi}_k)^\top\varepsilon_{k+1} && \\
    &\indenti{\Delta V_k =} - \frac{c(\gamma, \beta, \lambda_k)}{\N_k}\|(\bar{\Theta}_k - \hat{\Xi}_k)\phi_k\|^2 &&
\end{flalign*}
\begin{flalign*}
    &\indenti{\Delta V_k} = -\Big(a(\gamma, \beta, \lambda_k) - \frac{b^2(\gamma, \beta, \lambda_k)}{c(\gamma, \beta, \lambda_k)}\Big)\frac{\|\varepsilon_{k+1}\|^2}{\N_k} && \\ 
    &\indenti{\Delta V_k =} - \frac{c(\gamma, \beta, \lambda_k)}{\N_k}\Big\|(\bar{\Theta}_k - \hat{\Xi}_k)\phi_k - \frac{b(\gamma, \beta, \lambda_k)}{c(\gamma, \beta, \lambda_k)}\varepsilon_{k+1}\Big\|^2 &&
\end{flalign*}
Define
\begin{gather}
    d(\gamma, \beta, \lambda_k) = a(\gamma, \beta, \lambda_k) - \frac{b^2(\gamma, \beta, \lambda_k)}{c(\gamma, \beta, \lambda_k)}, \label{eqn:ht_d} \\
    c(\gamma, \beta) = \min_{\lambda_k \in [0, 1]} c(\gamma, \beta, \lambda_k), \label{eqn:ht_c_min} \\
    d(\gamma, \beta) = \min_{\lambda_k \in [0, 1]} d(\gamma, \beta, \lambda_k). \label{eqn:ht_d_min}
\end{gather}
Then,
\begin{equation} \label{eqn:ht_Lyapunov_increment_2}
    \Delta V_k \leq -d(\gamma, \beta)\frac{\|\varepsilon_{k+1}\|^2}{\N_k} \leq 0
\end{equation}
if $c(\gamma, \beta) > 0$ and $d(\gamma, \beta) > 0$.



Analytical expressions for values of $\gamma$ and $\beta$ under which $c(\gamma, \beta) > 0$ and $d(\gamma, \beta) > 0$ are difficult to obtain, and any such expression would be both cumbersome and not particularly illuminating.
We thus opt for a graphical approach in which we discretize $\gamma \in [0, 4]$, $\beta \in [0, 2]$, and $\lambda_k \in [0, 1]$, and at each value of $\gamma$ and $\beta$, calculate the minima of $c(\gamma, \beta, \lambda_k)$ and $d(\gamma, \beta, \lambda_k)$ over all $\lambda_k$. Finally, we alter the graphical approximation of $d(\gamma, \beta)$ to be negative at all points where $c(\gamma, \beta) \leq 0$. The resulting graphical approximation of the region of allowable hyperparameters is shown in Figure \ref{fig:ht_allowable_hyperparameters}. Note that the color map displays the value of $d(\gamma, \beta)$ whenever it is positive, but that the value is changed to $-\max_{\gamma, \beta} d(\gamma, \beta)$ wherever $d(\gamma, \beta) \leq 0$ in order to clearly show the border between the allowable and non-allowable regions. Allowable hyperparameters are thus any coordinate pair with coloring other than indigo.

Figure \ref{fig:ht_allowable_hyperparameters} clearly shows a much larger range of allowable hyperparameters than the range given in Proposition \ref{pro:ht_Lyapunov}. It was necessary to draw from this extended range of allowable hyperparameters in order to obtain the simulation results in Section \ref{sec:simulations}.

\begin{figure}
    \centering
    \includegraphics[width=0.5\textwidth]{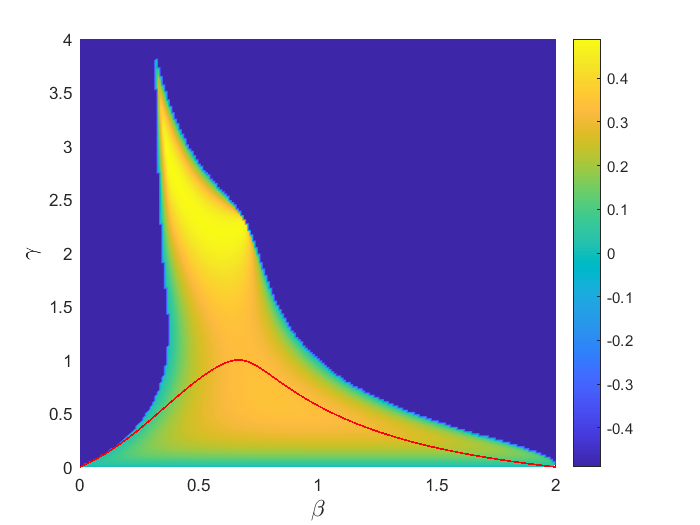}
    \caption{(to be viewed in color) The color map displays $d(\gamma, \beta)$ at all coordinates $(\gamma, \beta)$ wherever $c(\gamma, \beta) > 0$ and $d(\gamma, \beta) > 0$, and a constant negative value elsewhere. The region of allowable hyperparameters is thus the region of colors other than indigo. For each $\beta$, the red line indicates the maximum value of $\gamma$ allowed under Proposition \ref{pro:ht_Lyapunov}.}
    \label{fig:ht_allowable_hyperparameters}
\end{figure}

\end{document}